\documentclass[12pt]{article}
\usepackage{amssymb,amsmath,epsfig,amscd,eucal,psfrag,amsthm,relsize}
\usepackage{graphicx,hyperref}

\newtheorem{theorem}{Theorem}[section]

\newtheorem{lemma}[theorem]{Lemma}
\newtheorem{proposition}[theorem]{Proposition}

\newtheorem{corollary}[theorem]{Corollary}

\theoremstyle{definition}
\newtheorem{definition}[theorem]{Definition}

\theoremstyle{remark}
\newtheorem{remark}[theorem]{Remark}
\newtheorem{example}[theorem]{Example}

\newcommand{\B}{\mathbb{B}}

\newcommand{\G}{\mathcal{G}}

\newcommand{\F}{\mathbb{F}}

\newcommand{\R}{\mathbb{R}}
\newcommand{\Z}{\mathbb{Z}}

\newcommand{\wh}{\widehat}

\newcommand{\bigast}{\mathop{\mathlarger{\mathlarger{\mathlarger{*}}}}}

\newcommand{\curlyG}{\mathcal{G}}

\newcommand{\curlyH}{\mathcal{H}}

\input xy
\xyoption{all}

\title{Pro-$p$ subgroups of profinite completions of 3-manifold groups}
\author{Henry Wilton\footnote{Department of Pure Mathematics and Mathematical Statistics, Centre for Mathematical Sciences, Wilberforce Road, Cambridge, CB3 0WB, United Kingdom}~~and Pavel Zalesskii\footnote{Department of Mathematics, University of Bras\'ilia, 70910-9000 Bras\'ilia, Brazil}}

\begin{document}

\maketitle

\begin{abstract}
We completely describe the  finitely generated pro-$p$ subgroups of the profinite completion of the fundamental group of an arbitrary $3$-manifold. We also prove a pro-$p$ analogue of the main theorem of Bass--Serre theory for finitely generated pro-$p$ groups.
\end{abstract}

\section{Introduction}

In recent years there has been a great deal of interest in detecting properties of the fundamental group $\pi_1M$ of a $3$-manifold via its finite quotients, or more conceptually by its profinite completion.  This motivates the study of the profinite completion $\widehat {\pi_1M}$ of the  fundamental group of a $3$-manifold.  As in the case of finite groups, there is a Sylow theory for profinite groups, and pro-$p$ subgroups play the same central role in profinite group theory that $p$-subgroups play in the theory of finite groups.  Note that pro-$p$ subgroups of $\widehat {\pi_1M}$ are infinitely generated in general, so it is natural to begin the study from the finitely generated case. Amazingly enough,  it is possible to give a complete description of the  finitely generated pro-$p$ subgroups of the profinite completions of all 3-manifold groups, and this is one of the objectives of this paper (see Theorem \ref{main} below).

The work of this paper was partially motivated by \cite[Theorem F]{wilton_distinguishing_2017}, where it was proved  that if $M$ is hyperbolic then any finitely generated pro-$p$ subgroup of $\widehat{\pi_1M}$ is free.   In fact the proof of Theorem \ref{main} follows the same strategy as in \cite{wilton_distinguishing_2017}.   By applying first the (profinite) Kneser--Milnor decomposition and then the (profinite) JSJ decomposition, the theorem is reduced to the cases of Seifert-fibred and hyperbolic manifolds, possibly with cusps.  The main difficulty at this point is provided by hyperbolic manifolds with cusps.  In the closed case the result was proved in \cite{wilton_distinguishing_2017}.  In this case  (after passing to a finite-sheeted cover) there is a suitable hierarchy in which the corresponding actions on profinite trees are relatively profinitely acylindrical. However, we have to deal with cusps.

At this point we make heavy use of the profinite analogue of Bass--Serre theory for groups acting on trees. Note that this theory does not have the full strength of its classical original. The main theorem of Bass--Serre theory asserts that a group $G$ acting on a  tree $T$ is the fundamental group of a  graph of  groups $(\curlyG, T/G, D)$, where $D$ is a maximal subtree of $T/G$ and $\curlyG$ consists of edge and vertex stabilizers of a connected transversal of $T/G$. This does not hold in the profinite case: firstly $D$ does not always exist and secondly even if it exists it might not lift to $T$.

In this paper we prove a pro-$p$ analogue of the main theorem  of Bass--Serre theory for finitely generated pro-$p$ groups, that stands independently as a  contribution to the combinatorial theory of  pro-$p$ groups.

\begin{theorem} \label{intro:pro-pbass-serre}
Let $G$ be a finitely generated pro-$p$ group acting on a pro-$p$ tree $T$. Then $G$ is the fundamental pro-$p$ group of a profinite graph of
pro-$p$ groups $(\G, \Gamma)$.  Moreover, the vertex and edge groups of $(\G, \Gamma)$ are  stabilizers of certain vertices and edges of $T$ respectively, and stabilizers of vertices and edges of $T$ in $G$ are conjugate to subgroups of vertex and edge groups of $(\G, \Gamma)$ respectively.
\end{theorem}

We observe that $(\G,\Gamma)$ in this theorem is not $(\G, T/G)$ as in classical Bass--Serre theory.  To construct the graph of pro-$p$ groups  $(\G, \Gamma)$ we use a decomposition of $G$ as an inverse limit $G=\varprojlim_U G_U$ of finitely generated virtually free pro-$p$ groups $G_U$, and  their splittings  as the fundamental groups of finite graphs of finite $p$-groups obtained in \cite{herfort_virtually_2013}.

We then apply this theorem to describe the finitely generated pro-$p$ subgroups of the profinite completion of the fundamental group  of a cusped hyperbolic 3-manifold.

\begin{theorem}
Let $\pi_1M$ be the fundamental group of an orientable hyperbolic 3-manifold with cusps and $H$ a finitely
generated pro-$p$ subgroup of $\widehat \pi_1M$. Then $H$  is a free
pro-$p$ product of free abelian pro-$p$ groups of rank $\leq
2$.\end{theorem}

Analysing  $\widehat{\pi_1M}$ for Thurston's geometries case-by-case, we  obtain  the classification of finitely generated pro-$p$ subgroups of  $\widehat{\pi_1M}$ for an arbitrary  compact, orientable 3-manifold.

\begin{theorem} \label{main} A finitely generated pro-$p$ subgroup of the profinite completion $\widehat{\pi_1 M}$ of the fundamental group $\pi_1 M$ of a compact, orientable $3$-manifold $M$  is a free pro-$p$ product of the  pro-$p$ groups from the following lists of isomorphism types.

\begin{enumerate}

\item[(i)]  For $p>3$:   $C_p$; $\Z_p$;  $\Z_p\times \Z_p$;  the pro-$p$ completion of   $(\Z\times \Z)\rtimes \Z$; and the pro-$p$ completion of a residually-$p$ fundamental group of a non-compact Seifert fibred manifold with hyperbolic base orbifold;


\item[(ii)]  For $p=3$, in addition to the list of (i) of possible free factors we have a  torsion-free extension of $\Z_3\times\Z_3\times \Z_3$ by $C_3$.

\item[(iii)] For  $p=2$, in addition to the list in (i) we have the following free factors:
\begin{enumerate}

\item[(1)] the cyclic groups $C_{2^m}$;
\item[(2)] the dihedral groups $D_{2^k}$;
\item[(3)] the generalized quaternion groups $Q_{2^n}$;
\item[(4)] the infinite dihedral pro-$2$ group $\Z_2\rtimes C_2$;
\item[(5)] the torsion-free extensions of $\Z_2\times \Z_2\times \Z_2$ by one of the following finite $2$-groups: $C_2$, $C_4$, $C_8$, $D_2$, $D_4$, $D_8$, $Q_{16}$;
\item[(6)] the pro-$2$ extension of the Klein-bottle group $\Z\rtimes\Z$;
\item[(7)]   the pro-$2$ completion of all torsion free extensions of  a soluble  group  $(\Z\rtimes \Z)\rtimes \Z$ with a group of order at most $2$.
\end{enumerate}
\end{enumerate}
\end{theorem}

We note that, because the finitely generated pro-$p$ subgroups of profinite completions of 3-manifold groups are so restricted, they are not very useful for distinguishing 3-manifolds from each other. However, for the same reason, they are potentially useful for distinguishing the profinite completions of 3-manifold groups from the profinite completions of other discrete groups.

\subsection*{Acknowledgements}

The first author was supported by the EPSRC and the second author by Capes.  Much of the work for this paper was completed while the second author was a visiting fellow at Trinity College, Cambridge, and he acknowledges their hospitality.   Thanks are also due to the anonymous referee for pointing out a mistake in the first version of this paper.

\section{Preliminaries}

 In this section, we recall the necessary elements of the theory of profinite trees.

A \emph{graph} $\Gamma$ is a disjoint  union $E(\Gamma) \cup V(\Gamma)$
with two maps $d_0, d_1 : \Gamma \to V(\Gamma)$ that are the
identity on the set of vertices $V(\Gamma)$.  For an element $e$ of
the set of edges  $E(\Gamma)$, $d_0(e) $ is called the \emph{initial} and
$d_1(e) $ the \emph{terminal} vertex of $e$.

\begin{definition}
A \emph{profinite graph} $\Gamma$ is a graph such that:
\begin{enumerate}
\item $\Gamma$ is a profinite space (i.e.\ an inverse limit of finite
discrete spaces);
\item $V(\Gamma)$ is closed; and
\item the maps $d_0$ and $d_1$
are continuous.
\end{enumerate}
Note that $E(\Gamma)$ is not necessary closed.
\end{definition}

A \emph{morphism} $\alpha:\Gamma\longrightarrow \Delta$ of profinite graphs is a continuous map with $\alpha d_i=d_i\alpha$ for $i=0,1$.

By \cite[Prop.~1.7]{zalesskii_subgroups_1988} every profinite
graph $\Gamma$ is an inverse limit of finite quotient graphs of
$\Gamma$.

For a profinite space $X$  that is the inverse limit of finite
discrete spaces $X_j$, $[[\widehat{\mathbb{Z}} X]]$ is the inverse
limit  of $ [\widehat{\mathbb{Z}} X_j]$, where
$[\widehat{\mathbb{Z}} X_j]$ is the free
$\widehat{\mathbb{Z}}$-module with basis $X_j$. For a pointed
profinite space $(X, *)$ that is the inverse limit of pointed
finite discrete spaces $(X_j, *)$, $[[\widehat{\mathbb{Z}} (X,
*)]]$ is the inverse limit  of $ [\widehat{\mathbb{Z}} (X_j, *)]$,
where $[\widehat{\mathbb{Z}} (X_j, *)]$ is the free
$\widehat{\mathbb{Z}}$-module with basis $X_j \setminus \{ *
\}$ \cite[Chapter~5.2]{ribes_profinite_2010}.

For a profinite graph $\Gamma$, define the pointed space
$(E^*(\Gamma), *)$ as  $\Gamma / V(\Gamma)$, with the image of
$V(\Gamma)$ as a distinguished point $*$, and denote the image of $e\in E(\Gamma)$ by $\bar{e}$.  By definition  a profinite
tree  $\Gamma$ is a profinite graph with a short exact sequence
$$
0 \to [[\widehat{\mathbb{Z}}(E^*(\Gamma), *)]]
\stackrel{\delta}{\rightarrow} [[\widehat{\mathbb{Z}} V(\Gamma)]]
\stackrel{\epsilon}{\rightarrow} \widehat{\mathbb{Z}} \to 0
$$
where $\delta(\bar{e}) = d_1(e) - d_0(e)$ for every $e \in E(\Gamma)$ and $\epsilon(v) = 1$ for every $v \in V(\Gamma)$.  If $v$  and $w$ are elements  of a profinite tree   $T$, we denote by $[v,w]$ the smallest profinite subtree of $T$ containing $v$ and $w$ and call it a \emph{geodesic}.

A profinite graph is called a \emph{pro-$p$ tree} if one has the following exact sequence:

$$
0 \to [[\mathbb{F}_p(E^*(\Gamma), *)]]
\stackrel{\delta}{\rightarrow} [[\mathbb{F}_p V(\Gamma)]]
\stackrel{\epsilon}{\rightarrow} \mathbb{F}_p \to 0
$$
where $\delta(\bar{e}) = d_1(e) - d_0(e)$ for every $e \in E(\Gamma)$ and $\epsilon(v) = 1$ for every $v \in V(\Gamma)$. Note that any profinite tree is a pro-$p$ tree, but the converse is not true.  In fact, a connected profinite graph is a profinite tree if and only if it is a pro-$p$ tree for every prime $p$.

By definition, a profinite group $G$ \emph{acts} on a profinite graph
$\Gamma$ if  we have a continuous action of $G$ on the profinite
space $\Gamma$ that commutes with the maps $d_0$ and $d_1$.

 When we say that ${\cal G}$ is a \emph{finite graph of profinite groups}, we mean that it contains the data of the
underlying finite graph, the edge profinite groups, the vertex profinite groups and the attaching continuous maps. More precisely,
let $\Delta$ be a connected finite graph.  The data of a graph of profinite groups $({\cal G},\Delta)$ over
$\Delta$ consists of a profinite group ${\cal G}(m)$ for each $m\in \Delta$, and continuous monomorphisms
$\partial_i: {\cal G}(e)\longrightarrow {\cal G}(d_i(e))$ for each edge $e\in E(\Delta)$.

A finite graph of  groups $(\G,\Gamma)$ is said to be {\it
reduced}, if for every edge $e$ which is not a loop
neither $\partial_1\colon \G(e)\to \G(d_1(e))$ nor $\partial_0:
\G(e)\to \G(d_0(e))$ is an isomorphism; we call an edge $e$ which is not a loop and such that one of the edge maps $\partial_i$ is an isomorphism \emph{fictitious}. Any finite graph of  groups
can be transformed into a reduced finite graph of groups  by collapsing fictitious  edges using the
following procedure. If $e$ is a fictitious  edge, we can remove $\{e\}$ from the edge set of $\Gamma$, and
identify $d_0(e)$ and $d_1(e)$ to a new vertex $y$. Let
$\Gamma^\prime$ be the finite graph given by
$V(\Gamma^\prime)=\{y\}\sqcup V(\Gamma)\setminus\{d_0(e),d_1(e)\}$
and $E(\Gamma^\prime)=E(\Gamma)\setminus\{e\}$, and let
$(\G^\prime, \Gamma^\prime)$ denote the finite graph of  groups
based on $\Gamma^\prime$ given by $\G^\prime(y)=\G(d_1(e))$ if
$\partial_0(e)$ is an isomorphism, and $\G^\prime(y)=\G(d_0(e))$ if
$\partial_0$ is not an isomorphism. This procedure can be
continued until there are no fictitious edges. The resulting finite graph of groups
$(\overline \G,\overline \Gamma)$ is reduced.

\begin{remark}\label{reduction}  The reduction procedure just described does not change the fundamental group (as a group given by presentation), i.e. choosing a maximal subtree to contain the collapsing edge  the morphism $(\G,\Gamma)\longrightarrow (\G^\prime, \Gamma^\prime)$ induces the identity map on the fundamental group given by presentation by eliminating redundant relations associated with fictitious edges that are just collapsed by reduction. \end{remark}

The definition of the profinite fundamental  group of a connected
profinite graph of profinite groups is quite involved (see
\cite{zalesskii_fundamental_1989}). However, the profinite fundamental  group
$\Pi_1(\G,\Gamma)$ of a finite graph of finitely generated
profinite groups $(\G, \Gamma)$ can be defined as the profinite completion
of the abstract (usual) fundamental group $\Pi_1^{abs}(\G,\Gamma)$
(we use here that every subgroup of finite index in a finitely
generated profinite group is open \cite[Theorem 1.1 ]{nikolov_finitely_2007a}).  The fundamental profinite group
$\Pi_1(\G,\Gamma)$
has the following presentation:
\begin{eqnarray}\label{presentation}
\Pi_1(\mathcal{G}, \Gamma)&=&\langle \G(v), t_e\mid rel(\G(v)),
\partial_1(g)=\partial_0(g)^{t_e}, g\in \G(e),\nonumber\\
 & &t_e=1 \  {\rm for}\ e\in
T\rangle;
\end{eqnarray}
here $T$ is a maximal subtree of $\Gamma$ and
$\partial_0:\G(e)\longrightarrow
\G(d_0(e)),\partial_1:\G(e)\longrightarrow \G(d_1(e))$ are
monomorphisms.

In contrast to the abstract case, the vertex groups of $(\G,
\Gamma)$ do not always embed in $\Pi_1(\G,\Gamma)$, i.e.,
$\Pi_1(\G,\Gamma)$ is not always proper. However, the edge and
vertex groups can be replaced by their images in
$\Pi_1(\G,\Gamma)$ and after this replacement $\Pi_1(\G,\Gamma)$
becomes proper. Thus from now on we shall assume that
$\Pi_1(\G,\Gamma)$ is always proper, unless otherwise stated.

To obtain the definition of the fundamental pro-$p$ group of a finite graph of finitely generated pro-$p$ groups one simply replaces `profinite' by `pro-$p$' in the above definition.

The profinite (resp.\ pro-$p$) fundamental  group $\Pi_1(\G,\Gamma)$ acts on the standard profinite (resp.\
pro-$p$) tree $S$ (defined analogously to the abstract one)
associated to it, with vertex and edge stabilizers being conjugates
of vertex and edge groups, and such that
$S/\Pi_1(\G,\Gamma)=\Gamma$  \cite[Proposition 3.8]{zalesskii_subgroups_1988}.

\begin{example}\label{graph group completion} If $G=\pi_1(\G,\Gamma)$ is the fundamental group of a finite graph of (abstract) groups then one has the induced graph of profinite completions of edge  and vertex groups $(\widehat\G,\Gamma)$ (not necessarily proper) and  a natural homomorphism $G=\pi_1(\G,\Gamma)\longrightarrow \Pi_1(\widehat\G,\Gamma)$. It is an embedding if $\pi_1(\G,\Gamma)$ is residually finite. In this case $\Pi_1(\G,\Gamma)=\widehat{G}$ is simply the profinite completion.
Moreover,
 the standard tree $S(G)$
  naturally embeds in $S(\widehat G)$ if and only if  the edge and vertex groups $\G(e)$, $\G(v)$     are separable in  $\pi_1(\G,\Gamma)$, or equivalently $\G(e)$ are closed in $\G(d_0(e))$,   $\G(d_1(v))$    with
respect to the topology induced by the profinite topology on $G$
\cite[Proposition 2.5]{cotton-barratt_detecting_2013}.  In particular, this is the case  if vertex and edge groups are finitely generated and $G$ is subgroup separable.\end{example}

\medskip
In what follows, the notion of the fundamental pro-$p$ group of an infinite profinite graph of pro-$p$ groups is used only in Theorem \ref{pro-pbass-serre}, and in fact only in the form of the inverse limit of fundamental groups of finite graphs of finite $p$-groups.  We state this as a proposition that the reader can use as a definition. We need to define first a profinite graph of profinite groups. The definition is different from, but equivalent to, the one given in \cite{zalesskii_fundamental_1989}\footnote{The equivalence is an easy argument which we leave as an exercise to the reader.}.

\begin{definition} An inverse limit of  finite graphs of finite groups is called a \emph{profinite graph of profinite groups}.\end{definition}

\begin{proposition}\label{induces}  An inverse limit $(\curlyG,\Gamma)=\varprojlim_i (\curlyG_i,\Gamma_i)$ of finite graphs of profinite (resp. pro-$p$) groups  induces an inverse  limit of the fundamental profinite (resp. pro-$p$) groups $\varprojlim_i  \Pi_1(\curlyG_i,\Gamma_i)$ and $ \Pi_1(\curlyG,\Gamma)=\varprojlim_i  \Pi_1(\curlyG_i,\Gamma_i)$.\end{proposition}

\begin{proof} One uses Definition 9 (a) in \cite{serre_arbres_1977} with a base vertex (rather than a maximal subtree). Then by \cite[Proposition 2.4 ]{bass_covering_1993} a morphism of graphs of groups induces a homomorphism of fundamental groups, that in turn induces a homomorphism of profinite completions. So choosing a base vertex $v_0$ in $\Gamma$ and considering its image $v_{oi}$ in $\Gamma_i$ to be the base point in $\Gamma_i$, we have an inverse system   $\{\Pi_1(\curlyG_i,\Gamma_i)\}$.   The proof of $ \Pi_1(\curlyG,\Gamma)=\varprojlim_i  \Pi_1(\curlyG_i,\Gamma_i)$ is just a straightforward verification of the universal property from \cite[2.1]{zalesskii_fundamental_1989}.\end{proof}

\section{ Finitely generated pro-$p$ groups acting on profinite trees}

We now prove Theorem \ref{intro:pro-pbass-serre}.

\begin{theorem} \label{pro-pbass-serre} Let $G$ be a finitely generated pro-$p$ group acting on a pro-$p$ tree $T$. Then $G$ is the fundamental pro-$p$ group of a profinite graph of
pro-$p$ groups $(\G, \Gamma)$.
Moreover, the vertex and edge groups of $(\G, \Gamma)$ are  stabilizers of certain vertices and edges of $T$ respectively, and stabilizers of vertices and edges of $T$ in $G$ are conjugate to subgroups of vertex and edge groups of $(\G, \Gamma)$ respectively.
\end{theorem}
\begin{proof}
 For every open normal subgroup $U$ of $G$ consider
$\tilde U$, a subgroup generated by all intersections with vertex
stabilizers. Then by Proposition 3.5 and Corollary 3.6 of  \cite{ribes_pro-p_2000}, the quotient group $U/\tilde U$ acts freely on the
pro-$p$ tree $T/\tilde U$ and therefore it is free pro-$p$. Thus
$G_U:=G/\tilde U$ is virtually free pro-$p$. By \cite[Theorem 1.1]{herfort_virtually_2013} it follows that $G_U$ is the fundamental pro-$p$
group $\Pi_1(\mathcal{G}_U, \Gamma_U)$ of a finite graph of finite
$p$-groups with cyclic edge stabilizers.

 By Remark \ref{reduction} we have a morphism of graphs of groups $\eta:(\G_U,\Gamma_U)\longrightarrow (\overline\G_U,\overline\Gamma_U)$ to a reduced graph of groups. By  \cite[Corollary 3.3]{weigel_virtually_2017} the number of reduced graphs of groups $(\overline\G_U,\overline\Gamma_U)$ that can be obtained from $(\G_U,\Gamma_U)$ is finite.  Let
$\Omega_U$ be the set of such reduced graphs of groups. Since $G$ is finitely generated we can choose a linearly ordered system $\B$ of open normal subgroups $U$ that form a basis of neighbourhoods of $1$.

 By \cite[Proposition 1.10]{ribes_normalizers_2014} (see also its proof), for $V\subseteq U$ both in
$\B$ with $G/\tilde V$ written as the fundamental pro-$p$ group $G/\tilde V=\Pi_1(\overline \G_V,\overline\Gamma_V)$ of a reduced graph of $p$-groups, one has a natural decomposition of
$G/\tilde U$ as the pro-$p$ fundamental group $G/\tilde U=\Pi_1(\G_{V,U},\overline \Gamma_V)$
of a finite graph of finite $p$-groups $(\G_{V,U},\overline \Gamma_V)$, where the vertex and edge groups are
$\G_{V,U}(x)=\overline \G(x)\tilde U/\tilde U$, $x\in V(\overline \Gamma_V)\sqcup E(\overline\Gamma_V)$. Thus we have a
morphism $\nu_{V,U}\colon (\overline \G_V,\overline\Gamma_V)\longrightarrow (\G_{V,U},\overline\Gamma_V)$
of graphs of groups such that the induced homomorphism on the
pro-$p$ fundamental groups coincides with the canonical projection
$\varphi_{V,U}\colon G/\tilde V\longrightarrow G/\tilde U$. Choose a reduction morphism
$\eta_U\colon (\G_{V,U},\overline \Gamma_V)\longrightarrow (\overline \G_{V,U},\overline \Gamma_U)$
to a reduced graph of groups $(\overline \G_{V,U},\overline \Gamma_U)$ (it is not unique); it   induces the identity map on the fundamental group $G/\tilde U$  (see Remark \ref{reduction})  and so $\eta_U\nu_{V,U}$   induces the homomorphism $\Pi_1(\overline \G_V,\overline\Gamma_V)\longrightarrow \Pi_1(\overline \G_{V,U},\overline \Gamma_U)$  on
the pro-$p$ fundamental groups that coincides with the canonical
projection $\varphi_{UV}\colon G/\tilde V\longrightarrow G/\tilde U$. Observe that the morphisms $\nu_{V,U}$ and therefore $\eta_U\nu_{V,U}$ restricted on the second variable is just the collapse of fictitious edges. Thus, for $V\subseteq U$, one has a map
$\omega_{V,U}\colon \Omega_V\to\Omega_U$. Since $\B$ is linearly ordered this yields an inverse system (otherwise it would have not, since the reduction morphism is not unique).

 Hence
$\Omega=\varprojlim_{U}\Omega_U$ is non-empty. Let
$(\overline \G_U,\overline \Gamma_{U})\in\Omega$. Then
$(\overline \G,\Gamma)$ given by $(\overline \G,\Gamma)=\varprojlim
(\overline \G_U,\overline \Gamma_U)$,
is a profinite graph of  pro-$p$ groups
satisfying $G\simeq \Pi_1(\overline \G,\Gamma)$ (cf. Proposition \ref{induces}).
\end{proof}

\begin{corollary} \label{bound} Suppose the graphs $\Gamma_U$ in the proof of Theorem \ref{pro-pbass-serre}  are bounded   independently of $U$. Then $\Gamma$ is finite and so  $G$ is the fundamental group of a finite graph of pro-$p$ groups $(\G,\Gamma)$. \end{corollary}


It is not known whether $\Gamma$ can be made finite in Theorem \ref{pro-pbass-serre}. In the abstract case it comes automatically. One could achieve it if one could  bound the size of a  reduced finite graph of finite $p$-groups in terms of the minimal number of generators of its pro-$p$ fundamental group. In the case of cyclic edge groups this is possible.

\begin{corollary} \label{cyclic edge stabilizers} Let $G$ be a finitely generated pro-$p$ group acting on a pro-$p$ tree $T$ with cyclic edge stabilizers. Then $G$ is the fundamental pro-$p$ group of a finite graph of
pro-$p$ groups $(\G, \Gamma)$.
Moreover, the vertex and edge groups of $(\G, \Gamma)$ are  stabilizers of certain vertices and edges of $T$ respectively, and stabilizers of vertices and edges of $T$ in $G$ are conjugate to subgroups of vertex and edge groups of $(\G, \Gamma)$ respectively.\end{corollary}

\begin{proof}  By \cite[Lemma 2.2]{snopce_subgroup_2014} in the case of cyclic edge stabilizers the size of $\Gamma_U$ from the proof of Theorem \ref{pro-pbass-serre}  is bounded  independently of $U$, so the result follows from Corollary \ref{bound}. \end{proof}

\medskip

\begin{definition}\label{def: Profinite acylindrical}
The action of a profinite group $\widehat{\Gamma}$ on a profinite (or pro-$p$) tree $T$ is said to be \emph{$k$-acylindrical}, for $k$ a constant, if the set of fixed points of $\gamma$ has diameter at most $k$ whenever $\gamma\neq 1$.
\end{definition}

\begin{proposition}\label{virtfreeacylindrical} Let $G$ be a finitely generated pro-$p$ group acting $k$-acylindrically and virtually freely on a pro-$p$ tree. Then $G$ is a free pro-$p$ product of finite $p$-groups and a free pro-$p$ group.\end{proposition} 

\begin{proof}  Let $F$ be an open subgroup of $G$ acting freely on $T$. By \cite[Corollary 3.6]{ribes_pro-p_2000} $F$ is free pro-$p$. By \cite[Theorem 1]{herfort_virtually_2008} a virtually free pro-$p$ group whose torsion elements have finite centralizers satisfies the conclusion of the proposition. So we just need to show that torsion elements have finite centralizers. Let $g$ be a torsion element of $G$ and let $T^g$ be the subtree of fixed points (see \cite[Theorem 2.8]{zalesskii_subgroups_1988}. Then $C_G(g)$ acts on $T^g$. Since the action is virtually free  $|C_G(g)|=\infty$ would imply $|T^g|=\infty$ contradicting the hypothesis. Hence $C_G(g)$ is finite. This finishes the proof.\end{proof}

\begin{remark} 
The proposition is valid for second countable pro-$p$ groups. One just needs to replace the reference \cite{herfort_virtually_2008}  in the proof with \cite{macquarrie_second_2017}.
\end{remark}

For a normal subgroup $U$ of a profinite group $G$ acting on a profinite tree $T$ we denote by $\tilde U$ the subgroup of $U$ generated by all vertex stabilizers. Clearly $\tilde U$ is normal in $G$ and by \cite[Prop. 2.5]{zalesskii_subgroups_1988} $T/\tilde U$ is a profinite tree.

\begin{lemma}\label{acylindrical quotient} Let $G$ be a pro-$p$
 group acting $k$-acylindrically on a profinite tree $T$ and $U$ a closed normal subgroup of $G$. Then the action of $G_U=G/\tilde U$ on $T_U=T/\tilde U$ is $k$-acylindrical.\end{lemma}

\begin{proof} First note that a nontrivial  stabilizer of any edge
$e$ coincides with one of its vertex stabilizers $G_v$ or $G_w$,
say $G_w$,  since otherwise by \cite[Lemma 11.2]{wilton_distinguishing_2017}  the
stabilizers of two vertices of  $e$ do not generate a pro-$p$
group. It follows that a maximal connected subgraph $D$ of $T$ having non-trivial stabilizers in $G$ of all its edges has diameter at most $k$.

 We shall prove that the same is true in $G_U$; we use the subscript $U$ for images in $T/\tilde U$. Consider $1\neq g_U\in G_U$ that stabilizes a vertex in $T_U$.
 It suffices to  show that a maximal connected subgraph containing $v_U$ such that all edge stabilizers of it are non-trivial is contained in $D_U$. Let $e_U$ be an edge not in $D_U$ having a vertex $w_U\in D_U$. Let $g\in G$, stabilizing a vertex $v\in T$, be such that $g\tilde U/\tilde U=g_U$, and choose $D$ with $v\in D$.  Let $w$ be a vertex of $D$ whose image in $T_U$ is $w_U$. Then $G_w\neq \{1\}$ and there
exists an edge $e$ incident to $w$ whose image in $T_U$ is $e_U$.
But $G_e$ is trivial and so $(G_U)_{e_U}$ is trivial. Thus all edges not in $D_U$ with vertices in $D_U$ have trivial stabilizers in $G_U$, as required.
\end{proof}

Proposition \ref{virtfreeacylindrical} and Lemma \ref{acylindrical quotient}  enable us to generalize \cite[Theorem 11.1]{wilton_distinguishing_2017} from the 1-acylindrical case to the k-acylindrical case, for any natural number $k$.

\begin{theorem}\label{general acylindrical} Let $G$ be a finitely generated pro-$p$
 group acting $k$-acylindrically on a profinite tree $T$. Then $G$ is a free pro-$p$ product of vertex stabilizers and a free pro-$p$ group.\end{theorem}
 
 \begin{proof} Since $G$ is finitely generated, its Frattini series
$\Phi^n(G)$ is a fundamental system of neighbourhoods of 1. Let
$\widetilde{\Phi^n(G)}$ be  the subgroup of $\Phi^n(G)$ generated
be all vertex stabilizers of $\Phi^n(G)$. By Lemma \ref{acylindrical quotient},
$G_n=G/\widetilde{\Phi^n(G)}$  acts $k$-acylindrically  on a profinite tree
$T_n=T/\widetilde{\Phi^n(G)}$ and so by  Proposition \ref{virtfreeacylindrical}  is a free
pro-$p$ product
$$G_n=G/\widetilde{\Phi^n(G)}=\left(\coprod_{v} G_{nv}\right)\amalg F_{0n},$$
 of finite $p$ groups $G_{nv}$  and a free group  $F_{0n}$ in
$G_n$.

 Since $G=\lim\limits_{\displaystyle\longleftarrow} G_n$ and $G$ is
finitely generated, by choosing $n$ large enough we may assume that
the number of free factors is the same for every $m>n$, i.e.\ $v$
ranges over a finite set $V$. By
\cite[Theorem 4.2]{ribes_pro-p_2000} every finite subgroup of a free pro-$p$ product is
conjugate to a subgroup of a free factor. Therefore, the free
factors of $G_{m+1}$ are mapped onto the free factors of $G_m$ up
to conjugation. But in a free pro-$p$ product decomposition
replacing  any free factor by its conjugate does not change the
group. So, starting from $n$, we can inductively choose
$G_{m+1v}$ in such a way that its image in $G_m$ is $G_{mv}$. Let
$G_v$ be the inverse limit of $G_{mv}$. Then by
\cite[Lemma 9.1.5]{ribes_profinite_2010} $G= \coprod_{v\in V} G_v \amalg F_0$.

It remains to observe that $G_{mv}$ is a stabilizer of a vertex in
$T_m$ so the set of fixed points $T^{G_{mv}}$ is non-empty and
closed \cite[Theorem 3.7]{ribes_pro-p_2000}. Therefore, $G_v$
is the stabilizer of a non-empty set of vertices
$\lim\limits_{\displaystyle\longleftarrow} T^{G_{vm}} $. This
finishes the proof.  \end{proof}

We shall need these auxiliary lemmas in Section 4.

\begin{lemma} \label{pair edge stabilizers}
Let $H$ be a pro-$p$ group acting on a profinite tree $T$ with compact non-empty set of edges   having only finitely many maximal vertex stabilizers up to conjugation. Then $H$ acts on a profinite tree $D$ with non-empty set of edges such that, for every edge $e$ of $D$, the edge stabilizer $H_e$ fixes at least two edges in the original tree $T$.
\end{lemma}
\begin{proof} Let  $H_{v_1}, \ldots H_{v_n}$ be maximal stabilizers of vertices in $T$ up to conjugation. Then  the  stabilizer $H_m$ of any edge or vertex $m\in T$ is contained in some of $H_{v_i}$ up to $H$-translation of $m$. So assuming $H_m\leq H_{v_i}$  we see that $H_m$ fixes the geodesic $[m,v_i]$.  Since $E(T)$ is compact,  $H_m$ stabilizes an edge incident to $v_i$ whenever $m\neq v_i$ by \cite[Lemma 2.15]{zalesskii_subgroups_1988}.  So  if $m\not\in star(v_i)$, the stabilizer $H_m$ fixes at least two distinct edges. Now  $\bigcup_{i=1}^n Hstar(v_i)$  is a profinite subgraph of $T$ and, collapsing all its connected components  in $T$, by the Proposition on page 486 of \cite{zalesskii_open_1992}, we get a profinite tree
$D$ on which $H$ acts with  edge stabilizers each stabilizing at least two distinct edges of $T$ (as was just explained). Finally note that the connected components of $\bigcup_{i=1}^n Hstar(v_i)$ are stars, since  the geodesic $[gv_i,hv_j]$ is infinite whenever $gv_i\neq hv_j$ and $g,h\in H$. Indeed, otherwise choose a geodesic $[gv_i,hv_j]$ of minimal length with $gv_i\neq hv_j$ for $g,h\in H$.  To arrive at a contradiction we show that the group $K$ generated by the stabilizers of $gv_i$ and $hv_j$ is not pro-$p$. Note that $B=K[gv_i,hv_j]$ is a profinite tree with $B/K$ finite. Since $[gv_i,hv_j]$ is minimal, the stabilizers of all internal vertices of the geodesic are equal to the stabilizers of their incident edges.  Thus collapsing all but one edges of this geodesic and their $K$ translates, we may assume that $B/K$ has just one edge. Now we can apply \cite[Lemma 11.2 ]{wilton_distinguishing_2017}  to deduce that $K$ is not pro-$p$ as required.
\end{proof}

\begin{lemma} \label{almost cyclic edge stabilizers}  Let $G$ be a profinite group acting on a profinite tree $T$ with compact set of edges such that
 any non-trivial subgroup $K$ of $G$ stabilizing two edges $e_1, e_2$ is at most procyclic.
Let $H$ be a finitely generated pro-$p$ subgroup of $G$. Then  $H$  is the fundamental pro-$p$
group $\Pi_1(\curlyH,\Delta)$ of a finite graph of finitely generated  pro-$p$ groups with edge
 groups that are stabilizers of some pairs of distinct edges of $T$. In particular, if the
action of $G$ on $T$ is $1$-acylindrical then $H$ splits as a free pro-$p$ product
$H=\coprod_{v\in V(\Delta)} \curlyH(v) \amalg F$, where $F$ is a free pro-$p$ group
intersecting trivially all conjugates of vertex groups in $H$. \end{lemma}

 \begin{proof} 
 First note that a nontrivial  stabilizer of any edge $e$ coincides with one of its vertex stabilizers $H_v$ or $H_w$, say $H_w$,  since otherwise by  \cite[Lemma 11.2]{wilton_distinguishing_2017}  the stabilizers of two vertices of  $e$ do not generate a pro-$p$ group. This means that  connected components of the  abstract  subgraph whose edges have  non-cyclic $H$-stabilizers are at most edges.

Now since $H$ is finitely generated, its Frattini series
$\Phi^n(H)$ is a fundamental system of neighbourhoods of 1. Let
$\widetilde{\Phi^n(H)}$ be  the subgroup of $\Phi^n(H)$ generated
be all vertex stabilizers of $\Phi^n(H)$. Then
$H_n=H/\widetilde{\Phi^n(H)}$  acts on a profinite tree
$T_n=T/\widetilde{\Phi^n(H)}$ (cf.\   \cite[Proposition
2.5]{zalesskii_subgroups_1988}) and $\Phi^n(H)/\widetilde{\Phi^n(H)}$ acts freely on $T_n$ and
therefore is free pro-$p$  (see Theorem
2.6 \cite{zalesskii_subgroups_1988}). Note that the vertex and edge
stabilizers of $H_n$ acting on $T_n$ are finite epimorphic images
of the corresponding vertex and edge stabilizers of $H$ acting on
$T$ and so the images in $T_n$ of edges of $T$ with trivial edge
stabilizers  have  trivial stabilizers. Therefore, every abstract
connected component $\Omega_n$ of the subgraph of points of $T_n$ with
non-cyclic edge stabilizers in $H_n$  is exactly the image of  some abstract
connected components $\Omega$ of the subgraph of points of $T$ with
non-procyclic $H$-stabilizers in $T$. Indeed,
 suppose on the contrary there exists an edge $e_n\in \Omega_n$ not in the image of $\Omega$ but  having a vertex
$v_n$ in it. Let $v$ be a vertex of
$\Omega$ whose image in $T_n$ is $v_n$. Then $H_v$ is not procyclic and there
exists an edge $e$ incident to $v$ whose image in $T_n$ is $e_n$.
But $e\not\in \Omega$, so $H_e$ is procyclic and therefore $H_{e_n}$ is cyclic, contradicting the choice of $e_n$.

 By \cite[Lemma 8]{herfort_virtually_2008}, a virtually free pro-$p$ group has only finitely many finite subgroups up to conjugation. This means that $T_n$  has only finitely many edges $e_n$  up to translation with $(H_n)_{e_n}$ that are not cyclic. Hence the union of them is a profinite subgraph of $S$ and collapsing all its connected components  in $T_n$, by the
Proposition on page 486 in \cite{zalesskii_open_1992} we get a profinite tree
$D_n$ on which $H_n$ acts with cyclic edge stabilizers. Hence, by Corollary \ref{cyclic edge stabilizers},  $H_n$  is the fundamental pro-$p$ group of a
finite graph of finite $p$-groups $(\curlyH_n, \Gamma_n)$  with cyclic edge groups, where the vertex and edge groups of $(\curlyH_n, \Gamma)$  are stabilizers of
certain vertices and edges of $D_n$ respectively.  Moreover,  by collapsing edges that are not loops with a vertex group equal to the edge group  we may assume that $(\curlyH_n, \Gamma)$ is reduced (see Remark \ref{reduction}). Then by \cite[Lemma 2.2]{snopce_subgroup_2014}, the size of $\Gamma$ is limited in terms of the minimal number of generators $d(H)$ of $H$, i.e.\ is bounded independently of $n$.
Therefore, by Corollary \ref{bound}, $H$ is the fundamental group $\Pi_1(\curlyH,\Gamma)$ of a finite graph of pro-$p$ groups with cyclic edge stabilizers whose vertex and edge groups  are  stabilizers of certain vertices and edges of $T$. In particular,  it has only finitely many maximal stabilizers of vertices in $T$ up to conjugation. Then we can apply Lemma \ref{pair edge stabilizers} to deduce that $H$ acts on a profinite tree $D$ such that every edge stabilizer $H_e$, $e\in D$ fixes at least two edges in $T$. Now applying Corollary \ref{bound} again we deduce that $H$ is the fundamental group $\Pi_1(\curlyH,\Delta)$ of a finite graph of pro-$p$ groups satisfying the statement of the lemma.

Finite generation of the vertex groups $\curlyH(v)$  follow easily from the presentation of $\Pi_1(\curlyH,\Delta)$, the  finiteness of $\Delta$, and the fact that the edge groups $\curlyH(e)$ are procyclic.

To show the last statement, observe that in the case of a 1-acylindrical action the edge groups $\curlyH(e)$ of $(\curlyH, \Delta)$  are trivial (since each of them stabilizes a pair of distinct edges of $T$), so the result follows from the presentation (\ref{presentation}) of $\Pi_1(\curlyH, \Delta)$.
 \end{proof}

 We finish the section with two simple lemmas.

 \begin{lemma}\label{HNN}  Let $G=HNN(F(Y\cup X), C_i, t_i)$, $i=1,\ldots n$ be a pro-$p$ HNN-extension of a free pro-$p$ group $F=F(Y\cup X)$ on a finite set $Y\cup X$ such that each $C_i$  and $C_i^{t_i}$ are a conjugate of $\langle y\rangle$ in $F$  for some $y\in Y$. Then $G=\coprod_{y\in Y\setminus Z} (\langle y\rangle \times F(T_y)) \amalg F(Z\cup X)$ is a free pro-$p$ product, where $Z$ is a subset of $Y$ and $T_y$ is a  (possibly empty) subset of the set of stable letters $t_1,\ldots t_n$.\end{lemma}

  \begin{proof} We use induction on the cardinality of $Y$. If there exists $t_i$ such that $y_1^{f_1t_i}=y_2^{f_2}$ for some $y_1\neq y_2$ and $f_1,f_2\in F$ then we can rewrite $G$ as $G=HNN(F(Y\setminus{y_2} \cup (X\cup f_1t_if_2^{-1}), C_j, t_j)$, $j=1,\ldots i-1, i+1\ldots, n$ and the result follows from the inductive hypothesis.

 Thus we may assume that $C_i^{f_i}=C_i^{t_i}$ for some $f_i\in F$ for all $i$. Then replacing $t_i$ with $t_if_i^{-1}$ we may assume that $t_i$ centralizes $C_i$ for every $i$.  Moreover, conjugating $C_i$ in $F$, we may assume that every $C_i$ is generated by an element of $Y$. Now the result follows from the presentation  $G=HNN(F(Y\cup X), C_i, t_i)$, $i=1,\ldots n$.
 \end{proof}

 \begin{lemma}\label{over free factor}
 \begin{enumerate}

 \item[(i)] Let $H=H_1\amalg_{H_0} H_2$  be a pro-$p$ free amalgamated product  over a procyclic group $H_0$. Suppose $C_H(H_0)$ is abelian of rank at most 2 and $H_i=C_H(H_0)\amalg L_i$. Then $H=C_H(H_0)\amalg L_1\amalg L_2$.

 \item[(ii)] Let $H=HNN(H_1,H_0, t)$ be a pro-$p$  HNN-extension over a procyclic subgroup $H_0$. Suppose $C_H(H_0)$ is abelian of rank at most 2, $H_1=C_H(H_0)\amalg L_1$ and $H_0^t$ is  conjugate in $H_1$ either  to $H_0$ or to a subgroup of $L_1$. Then either $H=(\Z_p\times H_0)\amalg L_1$ or $H=\Z_p\amalg L_1$.
\end{enumerate}
  \end{lemma}

 \begin{proof} (i) Since $C_H(H_0)$ is abelian of rank at most 2, either $H_0$ is self-centralized in $H_1$ or it is self-centralized in $H_2$; say $H_0$ is self-centralized in $H_1$. Then $H=L_1\amalg H_2$  and the result follows.

 \medskip
 (ii)  If $H_0$ is  conjugate to  a subgroup of $L_1$ then  $H=L_1\amalg \langle t\rangle$. If   $H_0^t=H_0^{h_1}$ for some $h_1\in H_1$ then $H=L_1\coprod (\langle t h_1^{-1} \rangle \times H_0)$ so we are done as well.  \end{proof}

\section{Pro-p subgroups of profinite completions of 3-manifold groups}

\subsection{Reduction to the geometric cases}

Let $M$ be a compact, orientable, non-geometric 3-manifold.   Since every compact manifold is a retract of a closed manifold, it follows easily that we may reduce to the case in which $M$ is closed.  If $M$ is reducible then its Kneser--Milnor decomposition is non-trivial.  In particular, $\pi_1M$ acts on a tree with trivial edge stabilizers (i.e.\ 0-acylindrically), and with vertex stabilizers  exactly the conjugates of the reducible factors of $M$; such an action is 0-acylindrical.   Since 3-manifold groups are residually finite,  $\wh{\pi_1M}$ acts 0-acylindrically on a profinite tree, with vertex stabilizers the conjugates of the completions of the irreducible factors. Therefore, by Theorem \ref{general acylindrical}, we may reduce to the case where $M$ is irreducible.

If $M$ is non-geometric, we now argue similarly using the JSJ decomposition.  The action of an irreducible, non-geometric 3-manifold on its JSJ tree is 4-acylindrical \cite[Lemma 2.4]{wilton_profinite_2010}. The same is true for the action of the profinite completion on the profinite tree, as follows essentially from the results of \cite{wilton_profinite_2010} and \cite{hamilton_separability_2013}.  For completeness and ease of reference, we state the result here.

\begin{proposition}\label{prop: Profinite 4-acylindrical}
Let $M$ be a closed, irreducible 3-manifold.  Then the action of $\wh{\pi_1M}$ on the profinite tree associated to the JSJ decomposition of $M$ is 4-acylindrical.
\end{proposition}
 
This was stated in \cite[p.\ 376]{wilton_distinguishing_2017}, and a careful proof was written down in \cite[Proposition 6.8]{wilkes_profinite_2017b}.    Proposition \ref{prop: Profinite 4-acylindrical} follows in a straightforward manner from the following statement.

\begin{lemma}\label{lemma: Malnormal edge groups}
Let $N$ be geometric 3-manifold with incompressible toral boundary, not homeomorphic to an interval bundle over a torus or Klein bottle, and let $P_1,P_2$ be two $\mathbb{Z}^2$ subgroups conjugate to the fundamental groups of boundary components. Then, for any $\hat{\gamma}\in\widehat{\pi_1N}$, $\overline{P}_1\cap \overline{P}_2^{\hat{\gamma}}$ is equal to the maximal normal procyclic subgroup of $\widehat{\pi_1N}$.
\end{lemma}

Gareth Wilkes has pointed out that the results of \cite{wilton_profinite_2010} and \cite{hamilton_separability_2013} in fact prove a weaker statement than Lemma \ref{lemma: Malnormal edge groups}, since they only handle the case in which $\hat{\gamma}\in\pi_1N$.  However, similar techniques do indeed establish Lemma \ref{lemma: Malnormal edge groups}, and hence Proposition \ref{prop: Profinite 4-acylindrical}. For the Seifert fibred case, one argues as in \cite{wilton_profinite_2010}, appealing to \cite[Proposition 3.4 and Lemma 3.5]{chagas_limit_2007} instead of \cite[Lemma 3.6]{ribes_conjugacy_1996}. (Alternatively, a proof using more recent technology might appeal to \cite[Theorem 3.3]{wilton_distinguishing_2017}.)  For the hyperbolic case, one argues as in \cite{hamilton_separability_2013}, but replaces \cite[Lemma 4.7]{hamilton_separability_2013} by \cite[Lemma 4.5]{wilton_distinguishing_2017}.

With Proposition \ref{prop: Profinite 4-acylindrical}  in hand, by Theorem \ref{general acylindrical}, a finitely generated pro-$p$ subgroup of $\widehat{\pi_1M}$ is a free pro-$p$ product of pro-$p$ subgroups of the vertex groups, so it suffices to consider the geometric cases.

\medskip
\subsection{Hyperbolic manifolds} The closed hyperbolic case is the subject of \cite[Theorem F]{wilton_distinguishing_2017};  in this
case our subgroup is free pro-$p$.

\medskip

Thus we consider cusped hyperbolic 3-manifolds in this subsection.  Recall that a subgroup of $\pi_1M$  is called \emph{peripheral}  if it is conjugate to the fundamental group of a cusp and so is isomorphic to $\Z\times \Z$. A conjugate of the closure of a peripheral group in $\widehat{\pi_1M}$ will also be called peripheral; it is isomorphic to $\widehat \Z\times\widehat \Z$.

\begin{definition}
Suppose that a group $G$ is hyperbolic relative to a collection of parabolic subgroups $\{P_1,\ldots,P_n\}$.  A subgroup $H$ of $G$ is called \emph{relatively malnormal} if, whenever an intersection of conjugates $H^\gamma\cap H$ is not conjugate into some $P_i$, we have $\gamma\in H$.
\end{definition}

It is well known that $\pi_1M$ is hyperbolic relative to its peripheral subgroups.  Following Wise, we analyse the cusped hyperbolic 3-manifold $M$ by cutting along an embedded family of surface $\Sigma_i$ so every cusp is cut by some $\Sigma_i$. The following theorem is \cite[Theorem 9.1]{wilton_distinguishing_2017}.

\begin{theorem}\label{cusphiearachy}
A hyperbolic 3-manifold $M$ with cusps has a finite-sheeted covering space $N\to M$ that contains a disjoint family of connected, geometrically finite, incompressible subsurfaces $\{\Sigma_1,\ldots,\Sigma_n\}$ such that:
\begin{enumerate}
\item each cusp of $N$ contains a boundary component of some $\Sigma_i$;
\item each $\pi_1\Sigma_i\subseteq \pi_1 N$ is relatively malnormal.
\end{enumerate}
\end{theorem}

Cutting along the family of surfaces given by Theorem \ref{cusphiearachy} produces a graph-of-groups decomposition $(\G,\Delta)$ for $G=\pi_1N$, with the property that every vertex and edge group is hyperbolic and virtually special. Any peripheral subgroup $P_i$ leaves invariant an infinite (Tits')  axis and splits into direct product of infinite cyclic groups $P_i=L_i\times R_i$   such that $L_i$  fixes every edge of the axis and $R_i$ acts by translations. The stabilizer of any pair of distinct edges of the Bass--Serre tree $S$ is conjugate to such $L_i$ for some peripheral subgroup $P_i$.   Moreover, $\pi_1N$ is subgroup separable (see  \cite[Corollary 5.5]{aschenbrenner_manifold_2015}) and so by Example \ref{graph group completion}  we can  pass to the corresponding graph of profinite groups $(\widehat{\G},\Delta)$ such that $\widehat G=\Pi_1(\widehat{\G},\Delta)$ with $\widehat G$ acting on the standard profinite tree $S(\widehat G)$. By \cite[Theorem 4.2]{wilton_distinguishing_2017},  $\wh{\pi_1(\Sigma_i)}$ are relatively malnormal in the profinite completion of peripheral subgroups and so every non-trivial stabilizer of a pair of distinct edges of $S(\widehat G)$ is conjugate to a certain $\widehat L_i\leq \widehat P_i$.

\medskip

We shall need however a more specific hierarchy on vertex groups $\G(v)$ of $G$ given by the next proposition.

\begin{theorem}\label{special hierarchy}  Let $M$ be a cusped hyperbolic 3-manifold.   The fundamental group $\pi_1M$  possesses a finite index subgroup $G$ that admits a a graph-of-groups decomposition $(\G,\Gamma)$  such that:
\begin{enumerate}
\item every vertex and edge group is hyperbolic and virtually special, and the stabilizer of any pair of two distinct edges of the Bass--Serre tree $S$ is cyclic and is conjugate to a cyclic subgroup $L_i$ of a peripheral subgroup;
\item moreover, each vertex group $\G(v)$ has a malnormal hierarchy, which can be chosen so that, at each step, every $L_i$  belongs to a vertex group of the corresponding step-decomposition (so that at the last step of the vertex group being a free product of cyclic groups each $L_i$ is contained in some cyclic  free factor).
\end{enumerate}
\end{theorem}
\begin{proof}
We consider the finite-sheeted cover $N$ of $M$ provided by Theorem \ref{cusphiearachy}.  The fundamental group of $N$ comes equipped with a decomposition that satisfies item (1) above.  For item (2), we appeal to \cite[Theorem 2.11]{agol_alternate_2016}.  Indeed, let $N_v$ be a component of the given decomposition of $N$.  Then the $L_i$ that are conjugate into $N_v$ from a malnormal family $\mathcal{L}_v$, and so the pair $(\pi_1N_v,\mathcal{L}_v)$ is relatively hyperbolic by \cite[Theorem 7.11]{bowditch_relatively_2012}.  Since $\pi_1N_v$ is also hyperbolic and virtually special, we may apply  \cite[Theorem 2.11]{agol_alternate_2016} to obtain a normal subgroup $G_v$ of finite index in $\pi_1N_v$ so that the induced pair $(G_v,\mathcal{L}'_v)\lhd (\pi_1N_v,\mathcal{L}_v)$ admits, in the terminology of \cite{agol_alternate_2016},  a malnormal, quasiconvex, fully $\mathcal{L}'_v$-elliptic hierarchy.  This means precisely that there is a finite hierarchy in which every vertex group is hyperbolic and virtually special, and every edge group is malnormal and quasiconvex, and in which the intersection with every conjugate of $L_i\in \mathcal{L}_v$ is always elliptic, terminating in free products of groups from $\mathcal{L}'_v$ and free groups.

Since the subgroups $\pi_1N_v$ are separable in $\pi_1N$, we may pass to a finite-index normal subgroup $G$ of $\pi_1N$ so that every vertex group of the induced decomposition is contained in some $\pi_1N_v$.  The subgroup $G$ then has the properties claimed.
\end{proof}

We deduce from Lemma \ref{almost cyclic edge stabilizers} the following.

\begin{corollary} \label{parabolic in H}  Let $H$ be a finitely generated pro-$p$ subgroup of the profinite completion $\widehat G$. Then $H=(\curlyH, \Delta)$ is the fundamental pro-$p$ group of a finite graph of finitely generated pro-$p$ groups   with cyclic edge groups conjugate to subgroups of $L_i$, and where  the vertex  groups of $(\curlyH, \Delta)$  are conjugate to some vertex groups of $(\G, \Gamma)$.  \end{corollary}

\begin{corollary}\label{free vertex decomposition}
Let $H$ be a finitely generated pro-$p$ subgroup of the profinite completion of a vertex group $\widehat \G(v)$.  Let $\tilde{v}$ be a vertex in the standard tree $S(\wh{G})$ mapping to $v$; consider the action of $H$ on the set of incident edges at $\tilde{v}$, and let $e_1,\ldots,e_n$ be a set of orbit representatives, with corresponding stabilizers $\curlyH(e_i)$ in $H$.  Then $H$ splits as a free pro-$p$ product $H=\coprod_i  \curlyH(e_i) \amalg F$, where $F$ is a free pro-$p$ group. \end{corollary}

\begin{proof}    We use induction  on the hierarchy of Theorem \ref{special hierarchy} for $\G(v)$. So let $K$ be a subgroup that belongs to this hierarchy.

The base of the induction is the case in which $K$ is a free product of cyclic vertex groups $K(v)$. Since $\widehat P$ is malnormal in $\widehat{\pi_1M}$   (see \cite[Lemma 4.5]{wilton_distinguishing_2017}) and $\curlyH(e)= \widehat L_i^g$ for some $g\in \widehat G$,   in this case each non-trivial  intersection $K(v)\cap  \curlyH(e)=K(v)\cap  \widehat L_i^g=K(v)$  so each $\curlyH(e)= \widehat L_i^g$ generates a free factor.

Now, to perform the inductive step, note that  in this case by Lemma \ref{almost cyclic edge stabilizers} $H=\coprod_{v\in V(\Delta_K)} \curlyH(v)\amalg F$ is a free pro-$p$ product of  vertex  groups and a free pro-$p$ group $F$  intersecting trivially all conjugates of vertex groups. Applying the induction hypothesis to each vertex group $\curlyH_K(v)$ and using that a pro-$p$ free factor of a finitely generated pro-$p$ group is finitely generated, we deduce the statement.
\end{proof}

We are now ready to deduce the classification of pro-$p$ subgroups in the cusped hyperbolic case.

\begin{theorem} Let $\pi_1M$ be the fundamental group of an orientable hyperbolic 3-manifold with $k$ cusps, and $H$ a finitely
generated pro-$p$ subgroup of $\widehat{\pi_1M}$. Then $H$  is a free
pro-$p$ product of free abelian pro-$p$ groups of rank $\leq
2$.\end{theorem}

\begin{proof}  Let $G$ be the finite-index subgroup of $\pi_1M$ given by Theorem \ref{special hierarchy}. We prove first the theorem for $H\leq \widehat G$. By Corollary \ref{parabolic in H},  $H$ is the fundamental pro-$p$ group of a finite graph of finitely generated pro-$p$ groups $(\curlyH, \Delta)$, with cyclic edge groups conjugate to subgroups of $\widehat L_i$, where  the vertex  groups of $(\curlyH, \Delta)$  are stabilizers of certain vertices and edges of $S(\widehat G)$.     Moreover, by Corollary \ref{free vertex decomposition},  for each vertex $v\in \Delta$ we have $\curlyH(v)= (\coprod_{\curlyH_e\leq  \curlyH(v)}  \curlyH(e)) \amalg L_v$.

We use induction on the size of $\Delta$ to prove that $H$  is a free
pro-$p$ product of free abelian pro-$p$ groups of rank $\leq 2$, with every $H\cap \widehat L_i^g$ $g\in \widehat G$ being conjugate to a free factor.    If $\Delta$ is a vertex then the result just follows from the decomposition $\curlyH(v)= (\coprod_{\curlyH_e\leq  \curlyH(v)}  \curlyH(e)) \amalg L_v$ mentioned above.

Now let $e$ be an edge of $\Delta$.  Then  $H=H_1\amalg_{\curlyH(e)} H_2$ or $H=HNN(H_1, \curlyH(e), t)$, where  $H_1$ and $H_2$ are the fundamental groups of the graph of groups $(\curlyH, \Delta)$ restricted to the connected components  $\Delta_i$ of $\Delta\setminus \{e\}$ ($i=1,2$ o $i=1$). If $\curlyH(e)=1$ then the amalgamated free products  or HNN-extensions are free products and the result follows directly from the induction hypothesis.

Suppose $\curlyH(e)\neq 1$. By  the induction hypothesis
we may assume that $H_i=(\curlyH(e)\times H_{0i})\amalg L_i$.  Since  the profinite completion  $\widehat P$ of a parabolic subgroup $P$ is malnormal in $\widehat{\pi_1M}$ the centralizer of $\curlyH(e)$ in $H$ is abelian of at most rank $2$ so the result follows from Lemma \ref{over free factor}.

Now, to prove the general case, we observe that by what we just proved $H\cap \widehat G$ satisfies the statement of the theorem. Then the statement for $H$ follows from \cite[Theorem 1.2]{zalesskii_virtually_2016}.
\end{proof}

\subsection{Seifert fibred manifolds over hyperbolic bases}

In this section we deal with the cases in which $N$ admits a geometric structure modelled on either $\mathbb{H}^2\times\R$ or $\widetilde{SL_2(\R)}$.  Recall that, when $N$ is Seifert fibred, the fundamental group $\pi_1N$ is torsion-free of the form
\[
1\to\Z\to\pi_1N\to\pi_1O\to 1
\]
(not necessarily central) where $\pi_1O$ is a Fuchsian group. (We refer the reader to \cite{scott_geometries_1983} for details.)  Our classification theorem in this case shows that the pro-$p$ subgroups of the profinite completion of $\pi_1N$ are either free pro-$p$ products, or pro-$p$ completions of fundamental groups $\pi_1M$, where $M$ is a non-compact Seifert fibred manifold with residually $p$ fundamental group.

\begin{theorem} Let $G$ be the fundamental group a  Seifert fibred manifold with either $\mathbb{H}^2\times\R$ or $\widetilde{SL_2(\R)}$ geometry and $H$ a
finitely generated pro-$p$ subgroup of $\widehat G$. Then $H$ is a
cyclic   extension of a  free pro-$p$-product of cyclic
pro-$p$ groups.

\end{theorem}

\begin{proof}
Consider $H_0$, the image of $H$ in $\wh{\pi_1O}$, and note that $H_0$ is a finitely generated, pro-$p$ subgroup. It suffices to prove that $H_0$ is a  free pro-$p$-product
of cyclic pro-$p$ groups. Indeed, let $S$ be a surface subgroup of
finite index of $\pi_1O$.  Then $H\cap \widehat S$
is free pro-$p$ by \cite[Proposition 1.2]{zalesskii_profinite_2005}.

We now prove that the centralizer in $\wh{\pi_1O}$ of a torsion element of $H_0$ is finite.  Since no open subgroup of $\wh{\pi_1O}$ is pro-$p$, we may replace $O$ by a non-trivial finite-sheeted cover of degree $n>2$, where $n$ is coprime to $p$, whose completion contains $H_0$.  We may therefore assume that $\pi_1O$ is not a triangle group.  If $\pi_1O$ is torsion-free there is nothing to prove; otherwise, $\pi_1O$ splits as an amalgamation over an infinite cyclic group $\pi_1O=K*_C F$, where $F$ is free and $K=\bigast_i C_i$ is a free product of cyclic groups. It follows that $\widehat{\pi_1O}=\widehat K\amalg_{\wh{C}} \widehat F$ is a free profinite product with cyclic amalgamation. Applying \cite[Theorem 3.10]{zalesskii_subgroups_1988} twice, any finite subgroup of $\widehat{\pi_1O}$ is conjugate into $\wh{K}$, and thence to some $C_i$. By \cite[Theorem 3.12]{zalesskii_subgroups_1988}, $\widehat K \cap \widehat K^g$ is contained in a  $\wh{K}$-conjugate of $\wh{C}$. It follows that the centralizer $C_{\widehat{\pi_1O}}(c_i)$ of any element $c_i$ of $C_i$ is contained in $\widehat K$, and so by \cite[Theorem 9.1.12]{ribes_profinite_2010}, the centralizer coincides with $C_i$, as required.

By  \cite[Theorem 1]{herfort_virtually_2008}, virtually free pro-$p$ groups in which torsion elements have finite centralizers split as a free pro-$p$ product of
finite subgroups and a free pro-$p$ group.  Applying this to $H_0$ now completes the proof.
\end{proof}

\subsection{Sol and Nil manifolds}

Suppose that $N$ is a compact $3$-manifold with solvable fundamental group.  The groups that arise in this way were classified by Evans and Moser \cite{evans_solvable_1972}, and are all virtually polycyclic. It follows from this classification  that  a compact, orientable manifold $N$  with Nil or Sol geometry has a cover $N_0$ of index at most two such that $\pi_1{N_0}\cong\Z^2\rtimes \Z$, and we deduce the following.

\begin{theorem}
Let $N$ be a compact, orientable, Nil- or Sol-manifold.  If $p$ is odd then the Sylow $p$-subgroup of $\wh{\pi_1N}$ is isomorphic to the pro-$p$ completion of a semidirect product $\Z^2\rtimes\Z$, and for $p=2$, the Sylow $p$-subgroup has a subgroup of index at most 2 with this structure.
\end{theorem}
\begin{proof}
Let $H=\pi_1N_0$ as above. Denote by $O_p(G)$ the kernel of the epimorphism of a profinite group $G$ to its maximal pro-$p$ quotient.   Let $f: \widehat H\longrightarrow \widehat H/O_p(\wh{\Z}^2)$ be  the canonical map. Clearly $f$ restricts to an injection on $H$, so we identify $H$ with $f(H)$.  The image $f(\widehat \Z^2)$ is $\Z_p^2$ on which $\widehat H/O_p(\widehat H)$ acts by conjugation. Thus we have a homomorphism $\eta:\widehat H/O_p(\widehat H) \longrightarrow GL_2(\Z_p)$. Note that $GL_2(\Z_p)$ is virtually pro-$p$.   Hence the preimage $U$ of a $p$-Sylow subgroup of $\eta(H)$ is open in $H/O_p(H)$. Consider the intersection $K=U\cap \pi_1N$, a finite-index subgroup of $\pi_1N$.  The pro-$p$ completion $K_{\hat p}$ has a normal subgroup $\Z_p^2$ which is the closure of $\Z^2$ and the kernel of the natural map $K_{\hat p}\longrightarrow U$ intersects it trivially. One deduces then that this kernel is trivial, and the result follows.
\end{proof}

\subsection{The Euclidean case}

This case  consists of \emph{all} 3d Bieberbach groups -- i.e. torsion-free $\mathbb{Z}^3$-by-finite groups with the finite groups  of  the following 32 types \cite[Chapter 10.1, p.\ 794]{hahn_international_1987}: 27 cyclic and dihedral groups, 3  tetrahedral groups  $T, T_d, T_h$ and 2 octahedral groups $O$ and $O_h$. The orders of the these groups are divisors of $48$, so only the primes $p=2,3$ can lead to non-trivial Sylow subgroups. For $p=3$, all Sylow subgroups are $C_3$. See also  \cite[Lemma 4.9]{evans_solvable_1972}.

\medskip
Suppose $p=2$. Then we have the following list.

\begin{enumerate}

\item[1.] Cyclic groups: $C_2$, $C_4$, $C_8$.

\item[2.] Dihedral groups: $D_2$, $D_4$, $D_8$ of orders $4,8,16$ respectively.

 \item[3.] Tetrahedral groups $T, T_d, T_h$ have orders $12,24,24$ and their Sylow $2$-subgroups are $D_4$ and $D_8$, so we do not get anything new from the previous item.

\item[4.]  The two octohedral groups $O\cong S_4$ and $O_h$ of order $24$ and $48$ respectively. The Sylow $2$-subgroup of the first one is $D_4$, and of the second is the generalized quaternion group $Q_{16}$.

\end{enumerate}

 Thus we can deduce the following:

\begin{proposition} The Sylow $p$-subgroup of the profinite completion of the fundamental group of a Euclidean 3-manifold is a  pro-$p$ group  of one of the following types:
\begin{enumerate}
\item[(i)] $\Z_p^3$ for any $p$.
\item[(ii)] if $p=3$, $\Z_3^3$-by-$C_3$.
\item[(iii)]$p=2$, $\Z_2^3$-by-$C_2$, -$C_4$, -$C_8$, -$D_2$, -$D_4$, -$D_8$, -$Q_{16}$.
\end{enumerate}
\end{proposition}

 \begin{remark}
 Each extension may yield several isomorphism classes, so this does not give the number of groups.
 \end{remark}

There are the cases of $\Z$ and $D_\infty$.  We  do not get anything new, just  the pro-$p$ and pro-$2$ completions respectively.

\subsection{Interval bundles over Klein bottles}

The JSJ decomposition of a non-geometric 3-manifold $M$ may include pieces $N$ which are homeomorphic to twisted interval bundles over the Klein bottle, with fundamental group the Klein-bottle group $\Z\rtimes\Z$. These can be removed by passing to a double cover, but when $p=2$, the pro-2 completion of $\Z\rtimes\Z$ arises as the Sylow 2-subgroup of $\widehat{\Z\rtimes\Z}$.  We record this observation for later use.

\begin{proposition}\label{p:Klein bottle}
The Sylow $2$-subgroup of the profinite completion of the fundamental group of the twisted interval bundle over the Klein bottle is the pro-2 completion of $\Z\rtimes\Z$.
\end{proposition}

\subsection{Spherical manifolds}


We shall use here \cite[Theorem 2.4 ]{wolf_spaces_2011}.

\begin{theorem} (Theorem 2.4 \cite{wolf_spaces_2011} ) Finite fixed-point free subgroup of SO(4) belongs to the following list:

\begin{enumerate}

\item[1.]  Cyclic group $C_n$ of order $n$.

\item[2.] Dihedral groups $D_{2^k(2n+1)}$, for $k\geq 2$ and $n\geq 1$, with presentation
$\langle x,y\mid x^{2^k}=y^{2n+1}=1,xyx^{-1}=y^{-1}\rangle $ of order $2^k(2n+1)$.

\item[3.] Generalised quaternion groups $Q_{4k}$ of order $4k$.


\item[4.] The binary icosahedral group $I^∗=SL_2(\F_5)$ of order $120$.

\item[5.] The groups $P'_{8\cdot 3^k}=Q_8\rtimes C_{3^k}$  defined by the following presentation
$\langle x,y,z|x^2=(xy)^2=y^2,zxz^{-1}=y,zyz^{-1}=xy,z^{3^k}=1\rangle$ of order $8\cdot 3^k$.

\item[6.]  The direct product of any of the above group with a cyclic group of relatively prime order.
\end{enumerate}

\end{theorem}

\medskip
We shall list $p$-subgroups of all these groups. Since Case 6 does not produce anything new we analyse cases 1--5 only.

\medskip
Looking at the orders we see that for $p$ odd $p$-Sylow subgroups are cyclic. So we concentrate on the case $p=2$. Then we have only $3$ types of $2$-groups: cyclic, dihedrals $D_{2^k}$, generalized quaternions $Q_{ 2^n}$. Thus we deduce   the following

\begin{theorem} Let $G$ be the profinite completion of the fundamental group of a spherical $3$-manifold.

\begin{enumerate}

\item[(a)] If $p$ is odd, its Sylow $p$-subgroups are cyclic.

\item[(b)] If $p=2$, we have the following list of possible Sylow 2-subgroups:

\begin{enumerate}
\item[1.] cyclic;

\item[2.]  dihedrals $D_{2^k}$;

\item[3.]  generalized quaternions $Q_{2^n}$.

\end{enumerate}
\end{enumerate}
\end{theorem}


\bigskip

Putting the above results together we obtain Theorem \ref{main}.

\bibliographystyle{alpha}

\end{document}